\documentclass[11pt]{article}

\usepackage{amsfonts}
\usepackage{amssymb,amsmath,amsthm}
\usepackage{latexsym}
\usepackage{fullpage}
\usepackage{hyperref}

\newtheorem{maintheorem}{Theorem}

\newtheorem{theorem}{Theorem}[section]

\newtheorem{prop}[theorem]{Proposition}

\newtheorem{lemma}[theorem]{Lemma}
\newtheorem{cor}[theorem]{Corollary}

\theoremstyle{definition}

\newcounter{tenumerate}

\def\P{\mathbb{P}}

\newcommand{\one}{\1}
\newcommand{\deq}{\stackrel{\scriptscriptstyle\triangle}{=}}

\renewcommand{\epsilon}{\varepsilon}

\newcommand{\1}{\mathbf{1}}

\newcommand{\R}{{\mathbb R}}

\newcommand{\E}{{\mathbb E}}
\newcommand{\remove}[1]{}

\renewcommand{\leq}{\leqslant}
\renewcommand{\geq}{\geqslant}

\def\XXint#1#2#3{{\setbox0=\hbox{$#1{#2#3}{\int}$}
\vcenter{\hbox{$#2#3$}}\kern-.5\wd0}}

\begin{document}

\title{{\bf Distances in critical long range percolation}}

\author{Jian Ding\thanks{Partially supported by NSF grant DMS-1313596.}\\ University of Chicago \and Allan Sly\thanks{Partially supported by a Sloan Research Fellowship in mathematics and by NSF award DMS-1208339.} \\ UC Berkeley}
\maketitle

\begin{abstract}
We study the long range percolation model on $\mathbb{Z}$ where sites $i$ and $j$ are connected with probability  $\beta |i-j|^{-s}$.
Graph distances are now well understood for all exponents $s$ except in the case $s=2$ where the model exhibits non-trivial self-similar scaling.
Establishing a conjecture of Benjamini and  Berger~\cite{BenBer:01}, we prove that the typical distance from site $0$ to $n$ grows as a power law $n^{\theta(\beta)}$ up to a multiplicative constant for some exponent $0<\theta(\beta)<1$ as does the diameter of the graph on a box of length $n$.
\end{abstract}

\section{Introduction}


The long range percolation model (LRP) is a model of percolation on $\mathbb{Z}$, where edges between sites $x,y\in \mathbb{Z}$ are placed independently with probability $\mathtt{P}(|x-y|)$ where $\mathtt{P}(r)\sim\beta r^{-s}$ for parameters $s,\beta>0$.  The geometric behavior of such percolation clusters is now well understood in terms of the typical graph distances and diameters~\cite{BenBer:01,Biskup:04,CGS:02} with the notable exception of $s=2$.  For this exponent, however, the model has a natural self-similar structure which ensures that edges on all scales are important in determining distances.  Previously, only non-matching polynomial bounds were known for a restricted range of $\beta$~\cite{CGS:02}.

We will consider geometric properties of the cluster in the case of $s=2$.  To ensure that an infinite component exists we will set $\mathtt{P}(1)=1$  so that all of the nearest neighbour edges present.  The scaling of distances is quite sensitive to the particular form of $\mathtt{P}$  so we set $\mathtt{P}(r) = \beta r^{-2} \wedge 1$ when $r>1$.

We denote by $d^*_{\mathrm{LRP}}(x, y)$ the random metric generated by the LRP (i.e., the graph distance), where each edge, both nearest neighbor and long edges, are counted as weight 1. We further consider a constrained distance $d_{\mathrm{LRP}}(x, y)$ to be the distance from $x$ to $y$ on the subgraph induced by the vertices in $[x,y]$. Finally, we denote by $\mathrm{Diam}_{\mathrm{LRP}}(x, y)$ the diameter of the subgraph induced by the vertices in $[x,y]$. Much of the difficulty in this model is that these distances are not concentrated but are tight when properly renormalized. We thus denote asymptotic scaling in probability by $A_n\asymp_P B_n$ if for all $\epsilon>0,\exists c, C>0$ we have that $\P(cB_n \leq A_n\leq CB_n) \geq 1-\epsilon$ for all $n\in \mathbb{N}$.
Our main result is that all these quantities scale asymptotically as a power-law whose exponent is a function of $\beta$.
\begin{maintheorem}\label{thm-1}
For all $\beta>0$, there exists $\theta= \theta(\beta) \in (0, 1)$ such that
\[
d^*_{\mathrm{LRP}}(0, n) \asymp_P d_{\mathrm{LRP}}(0, n) \asymp_P \mathrm{Diam}_{\mathrm{LRP}}(0, n) \asymp_P n^\theta\,.
\]
\end{maintheorem}

Our main tool for studying long range percolation will be the  continuous analogue of LRP on~$\mathbb{R}$. Here the set of edges $\mathcal{E}$ is given by a Poisson point process where edges between $x$ and $y$ occur at intensity $\beta |x-y|^{-2}$ with respect to Lesbegue measure on $\mathbb{R}^2$. The presence of an edge between $x$ and $y$ is equivalent to identifying those points so that they are at distance 0.  To ensure that the metric isn't reduced to 0, we truncate the set of edges to those whose length is in the interval $(\delta, \delta')$.  We use the notation $\langle\cdot,\cdot \rangle$ to denote a long edge. By analogy with LRP, we define $d^*_{(\delta, \delta')}(\cdot, \cdot)$ to be the random metric generated by the continuous model. That is,
\begin{equation}\label{eq-def-cts-unconstrained}
d^*_{(\delta, \delta')}(x, y) : = \min_{m \geq 0} \Big\{ |x - u_1|   + \sum_{i=1}^{m-1}|v_i-u_{i+1}|  + |v_m - y|:  \{\langle u_i,v_i\rangle\}_{1\leq i \leq m}\subseteq \mathcal{E}, |u_i - v_i| \in (\delta, \delta')  \Big \} \,,
\end{equation}
where the case of $m=0$ evaluates to $|x-y|$.  This minimization can be  interpreted as over geodesics crossing the edges $\langle u_i,v_i \rangle$ in that order, where the distance corresponds to the sum of the distances between the end point of one long edge and the starting point of the next.  The main object we consider will be a constrained distance $d_{(\delta, \delta')}(x, y)$, where the geodesic must remain within $[x,y]$. Formally, we define $d_{(\delta, \delta')}(x, y)$ as
\begin{equation}\label{eq-def-cts}
\min_{m \geq 0} \Big\{ |x - u_1|   + \sum_{i=1}^{m-1}|v_i-u_{i+1}|  + |v_m - y| :  \{\langle u_i,v_i\rangle\}_{1\leq i \leq m}\subseteq \mathcal{E}, |u_i - v_i| \in (\delta, \delta') , u_i,v_i\in [x,y] \Big \} \,.
\end{equation}
The diameter, $\mathrm{Diam}_{(\delta,\delta')}(x, y)$ is defined as the maximum distance between two points in $[x,y]$ using only edges within $[x,y]$.  In order to derive Theroem~\ref{thm-1}, the following theorem is a major ingredient, to which most of the paper is devoted.
\begin{maintheorem}\label{thm-cts}
For all $\beta>0$, there exists $\theta= \theta(\beta) \in (0, 1)$ such that
\[
d^*_{(1, n)}(0, n) \asymp_P d_{(1, n)}(0, n) \asymp_P \mathrm{Diam}_{(1, n)}(0, n) \asymp_P n^\theta\,.
\]
\end{maintheorem}

Our methods do not give an explicit formula for $\theta(\beta)$. Instead, it arises from a Sub-additive Ergodic Theorem argument on the expected distance between points in the constrained distance $d_{(1, n)}(0,n)$ which we show is sub-multiplicative up to constants.  The bulk of the work then falls to showing that it is also super-multiplicative up to constants.  We first establish that for a suitable range of ``good'' $n_i$  we can control the second moment $d_{(1, n_i)}(0,n_i)$ in terms of the first. Thus, we could apply the second moment method to get a lower bound on the probability that the distance is of the same order as its expectation.  By considering such distances across a range of scales we can boost this probability  arbitrarily close to 1.  This finally allows us to take a union bound over paths and establish the required super-multiplicativity.

A key tool in each of these steps is an enumeration over possible paths which allows us to take union bounds  given suitable large deviation estimates. For controlling the second moment of the distance in terms of the first, we  break an interval into a collection of sub-intervals. Then we show that with good probability we can avoid the worst sub-interval and use moment inequalities of order statistics.  We complete the proof by coupling the continuous model with the discrete long range percolation model to derive the required estimates for Theorem~\ref{thm-1}.

\subsection{Background}

Long range percolation was first considered in the 1980's as a one dimensional model which can exhibit phase transition.  The initial question addressed involved the existence of an infinite component when the connection probabilities are all strictly less than one.  Schulman~\cite{Schulman:83} showed that there is no infinite component when $s>2$ while Newman and Schulman~\cite{NewSch:86} showed that if $s<2$ or if $s=2$ and $\beta$ is large and $\mathtt{P}(1)$ is increased to sufficiently close to 1 then there is percolation.  The case $s=2$ proved to be  the most delicate, in which  Aizenman and Newman~\cite{AizNew:86} showed that $\beta=1$ is the critical value with no  percolation for $\beta\leq 1$ and percolation for $\beta>1$ provided that $\mathtt{P}(1)$ is sufficiently close to 1.

More recently more geometric properties were considered, particularly due to the interest in ``small world'' networks (see e.g.~\cite{Kleinberg:00,Milgram:67,WatStr:98}).   In this line of research, it is assumed that $\mathtt{P}(1) = 1$ in order to guarantee the connectedness. Benjamini and Berger~\cite{BenBer:01} initiated the study of the scaling of distances proving bounds in a range of regimes and conjecturing the appropriate scaling in each regime.  When $d=1$ these have been verified in every case except $s=2$.  When $s<1$, it was shown by Benjamini, Kesten, Peres, and  Schramm \cite{BKPS:04} that the diameter on the block of vertices $\{0,\ldots,n\}$  is bounded by a constant with high probability and the typical degree grows polynomially.  In the case $s=1$ the typical distance and diameter is of order $\frac{\log n}{\log\log n}$, as proved by Coppersmith, Gamarnik, and  Sviridenko \cite{CGS:02}.  Using a multiscale construction, Biskup~\cite{Biskup:04,Biskup:10} showed that when $1< s < 2$ the diameter grows polylogarithmically as $(\log n)^{\delta+o(1)} $ for an explicit $\delta=\frac{\log 2}{\log (2/s)}$.  Finally when $s>2$ the extra edges are too sparse to significantly affect the diameter which grows linearly in $n$~\cite{BenBer:01}.

In the case of $s=2$, a polynomial upper bound of $n^{\theta_2(\beta)}$ on the diameter and a polynomial lower bound of $n^{\theta_1(\beta)}$ when $\beta<1$, was proved in \cite{CGS:02}.  When $\beta<1$ it is straightforward to verify that with high probability there are at least $n^{1-\beta-o(1)}$ edges $(i,i+1)$ such that there is no other connection from $[0,i]$ to $[i+1,n]$ hence providing a lower bound on the diameter.  We extend this lower bound to all $\beta$ and, moreover, show that the growth is given by a single exponent $\theta(\beta)$.

\subsection{Open Problems}

One might ask if more can be said about the metric structure  of long range percolation.  There has been considerable interest recently in scaling limits of large random graphs such as critical random graphs~\cite{ABG:12}, minimal spanning trees~\cite{ABGM:13} and random planar maps~\cite{LeGall:11,Miermont:11}.  When $s=2$ the self-similar structure (for the continuous analog)   and  polynomial scaling of LRP suggests a scaling limit in the Gromov-Hausdorff sense to a random self-similar metric space.

Another line of inquiry is on the evolution of random walks on such  networks.  Benjamini,  Berger, and  Yadin~\cite{BBY:08} showed that when $1\leq s <2$, the spectral gap on LRP on a cycle of length $n$ scales like $n^{-(s-1)}$ and then jumps discontinuously to $n^{-2}$ for $s>2$.  It is natural to expect that the correct scaling when $s=2$ is $n^{\chi(\beta)}$  interpolating between $n^{-1}$ and $n^{-2}$.

Finally, one may ask about the scaling limit of random walks on the LRP cluster.  It was shown by Crawford and Sly \cite{CraSly:09} that when $1<s<2$ the scaling limit is $(s-1)$-stable motion while when $s>2$ it is Brownian motion.  It is natural to suspect that when $s=2$ there  is a joint scaling limit of the metric and the walk to a self-similar metric space together with a diffusion on the space.

\subsection{Organization of the paper}
In Section~\ref{sec:explore} we describe our exploration process for enumerating over paths in the graph.  In Section~\ref{s:firstmoment} we prove the submultiplicativity of the expected distance and show that the resulting power law is strictly between 0 and 1 for all $\beta$.  Then in Section~\ref{s:SecondMoment} we bound the second moment of distances and establish super-multiplicativity from which we deduce Theorem~\ref{thm-cts}.
Finally in Section~\ref{s:ContinuousToDiscrete} we couple the continuous and discrete models establishing Theorem~\ref{thm-1}.

\subsection{Acknowledgements}
The authors would like to thank Nick Crawford, Noam Berger and Marek Biskup for useful conversations. This work was initiated while J. D. was visiting MSRI.

\section{An exploration process argument} \label{sec:explore}
We first demonstrate some a priori properties on potential geodesics in the long range percolation by an argument in the flavor of an exploration process. For an integer $m$, denote by $[m]= \{1, 2, \ldots, m\}$. For a path $P = z_0, \ldots, z_m$ in the metric $d_{(\delta, \delta')}(\cdot, \cdot)$, let
\[
I = I_P := \{  i\in [m]: \exists\mbox{ an edge between } z_{i-1} \mbox{ and } z_{i}, |z_{i-1} - z_{i}| \in [\delta, \delta']\}\,.
\]
Then (in view of \eqref{eq-def-cts}) the length of $P$ is defined to be $\|P\|_{1} \deq \sum_{i\in [m]\setminus I} |z_{i-1} - z_{i}|$. In addition, for $i\in I$, we define the length of the edge connecting $z_{i-1}$ and $z_{i}$ to be $|z_{i-1} - z_{i}|$ (note that this length is not counted in measuring $\|P\|_1$), and for $i\not\in I$ we say that $(z_{i-1}, z_{i})$ is a gap.  We say a path is \emph{proper} if it does not contain two consecutive gaps and does not reuse a jump.  We say that two proper paths are \emph{equivalent} if they have the same starting point and make the same set of jumps in the same order i.e they  are equal except for the final endpoint.  A shortest path will be called a \emph{geodesic}.

For $t\geq 0$, denote by $\mathcal{P}_t$ the collection of equivalence classes of proper paths starting from the origin $o$ whose length is between $t/2$ and $t$ in metric $d_{(1, \infty)}$. For $P\in \mathcal{P}_t$, denote by $h(P) = |I|$ the number of hops (long range edges) in the path $P$.  Furthermore, for $\epsilon>0$ denote by $\mathcal P_{t, \epsilon}\subset \mathcal P_t$ the collection of equivalence classes of paths $P$ such that the number of gaps whose lengths are at least $\epsilon$ is less than or equal to $\epsilon t$.

\begin{lemma}\label{lem-exploration}
For all $\beta>0$, there exists $C>0$ such that for all $t>0$ and $\alpha\geq C$
 $$\E |\mathcal P_t| \leq C^t \mbox{ and }\P(\exists P\in \mathcal{P}_t: h(P) \geq \alpha t) \leq (C/\alpha)^t\,.$$
 Furthermore, for $\epsilon>0$, we have
 $$\E |\mathcal P_{t, \epsilon}| \leq C\mathrm{e}^{C\sqrt{\epsilon}t} \,.$$
\end{lemma}
\begin{proof}
For $n\in \mathbb{N}$, consider a discretized model on $\mathbb{Z}/n=\{i/n:i\in\mathbb{Z}\}$, where a long edge is placed between $i/n$ and $j/n$ for $i-j\geq n$ with probability $\beta(i-j)^{-2} \wedge 1$ (in the special case that $|i-j|=1$, the edge connecting $i/n$ and $j/n$ presents with probability 1), independently. It is clear that the discretized model converges to our continuous model $d_{(1, \infty)}$, and thus it remains to prove the desired bounds in the discretized model uniformly for all sufficiently large $n$.  In the continuous model no vertex can be the endpoint of two jumps so we may restrict our attention to paths with the same property.

We count paths by encoding them as a set of instructions for an explorer.  For each $x\in \mathbb{N}/n$, we associate a tuple $(w_x, j_x)$ with $w_x, j_x\in \mathbb{Z}/n$ which will serve as instructions on where the explorer shall move. Our exploration process  proceeds as follows. Start the explorer at the origin (set $s_0 = 0$). For $i\geq 0$, reveal the tuple $(w_{s_i}, j_{s_i})$ and have the explorer walk distance $w_{s_i}$  (using the nearest neighbor edges in $\mathbb{Z}/n$) and then jump $j_{s_i}$ distance (presumably using long edges in the discretized model); set $s_{i+1} = s_i + w_{n_i} + j_{n_i}$. We stop the process once $\sum_i |w_i| \geq t$. Clearly, the trace of the explorer forms a candidate path in the long range percolation graph. Thus, $\E |\mathcal{P}_t|$ is bounded by the sum of the probability for a trace to be a legal path in long range percolation over all possible traces (i.e. the probability that the edges in the path are actually present). Write $c^* = 4\sum_{i=1}^\infty i^{-2}$. Denoting by index $k$ the number of hops in the path, we obtain from straightforward combinatorial computation that for a large constant $C = C(\beta)>0$ and uniformly for all sufficiently large $n$,
\begin{align}
\E |\mathcal{P}_t|& \leq \sum_{k=0}^\infty \, \, \sum_{\forall i:\mathrm{sgn}(w_{n_i}), \mathrm{sgn}(j_{s_i})\in \{\pm1\}} \,\,\sum_{\sum_{i=0}^{k-1}|w_{s_i}| \leq t} \,\, \sum_{\forall i: j_{s_i}\geq 1} \prod_{i=0}^k \frac{\beta}{j_{s_i}^2 n^2}\nonumber\\
&\leq 2\sum_{k=0}^\infty  \binom{n t + k}{k} \left(\frac{c^*\beta}{n}\right)^k \label{eq-referee-1}\\
&\leq 2 \sum_{k=0}^{nt} \frac{(2\mathrm{e}nt)^k}{k^k} \left(\frac{c^*\beta}{n}\right)^k + 2\sum_{k> nt} \frac{(2\mathrm{e}k)^{nt}}{(nt)^{nt}} \left(\frac{c^*\beta}{n}\right)^k \leq C^t\,,\nonumber
\end{align}
where the initial summation is over the direction of the walks and jumps and their magnitudes and the probability that the necessary edges are present and the first inequality comes from summing over the choices of jump lengths.
For the final we have used the fact that $\frac{(2\mathrm{e}nt)^k}{k^k} \left(\frac{c^*\beta}{n}\right)^k$ decays geometrically for $k\geq Ct$ and that the second sum is negligible for $n \geq 10^4 c^* \beta$.

Now, note that $\P(\exists P\in \mathcal{P}_t: h(P) \geq \alpha t)\leq \E|\mathcal{P}_t \cap \{P: h(P)\geq \alpha t\}|$. Therefore, an analogous derivation (to that of $\E |\mathcal{P}_t|$) yields that
\begin{align*}\P(\exists P\in \mathcal{P}_t: h(P) \geq \alpha t) &\leq  \sum_{k\geq \alpha t}\, \, \sum_{\forall i:\mathrm{sgn}(w_{n_i}), \mathrm{sgn}(j_{s_i})\in \{\pm1\}} \,\,\sum_{\sum_{i=0}^{k-1}|w_{s_i}| \leq t} \,\, \sum_{\forall i: j_{s_i}\geq 1} \prod_{i=0}^k \frac{\beta}{j_{s_i}^2 n^2}\\
& \leq 2 \sum_{k=\alpha t}^{nt} \frac{(2\mathrm{e}nt)^k}{k^k} \left(\frac{c^*\beta}{n}\right)^k + 2\sum_{k> nt} \frac{(2\mathrm{e}k)^{nt}}{(nt)^{nt}} \left(\frac{c^*\beta}{n}\right)^k \leq (C/\alpha)^t\,,
\end{align*}
for $\alpha \geq C \geq 8 \mathrm{e}^2 c^* \beta$. Finally, we estimate $\E |\mathcal{P}_{t, \epsilon}|$.  We assume that $\epsilon$ is  small, otherwise we  use our trivial upper bound on $|\mathcal{P}_t|$.  We let $M_{m, \ell, r, \theta}$ be the number of ways of partitioning the interval $[0, m]$ into $\ell$ segments with integer lengths such that there are $r$ segments of length at least $\theta$. Then, indexing by $i$ the number of end points formed by the $r$ long segments, assigning these and then choosing the remaining points we get
\begin{align*}
M_{m, \ell, r, \theta} &\leq \sum_{i=r}^{2r\wedge (\ell-1)} \binom{m}{i} \binom{m \wedge [(\ell - r)\theta]}{\ell - i - 1}\\
&\leq \sum_{i=r}^{2r\wedge (\ell-1)} \frac{m^i}{i!} \frac{(m \wedge[(\ell - r)\theta])^{\ell - i - 1}}{(\ell - i - 1)!} \leq \frac{2^\ell }{(\ell-1)!} m^{2r\wedge (\ell-1)} ((\ell - r)\theta)^{(\ell - 2r-1)\vee 0}\,.
\end{align*}
We extend the final gap to make the total length $t$ and thus possibly increasing the number of segments of length at least $\epsilon n$  by 1. Fix a constant $c_1>2$ and decompose the number of hops $k$ into three intervals $[0, c_1 \epsilon t]$, $(c_1\epsilon t, nt]$ and $(nt, \infty)$. Writing $\E |\mathcal P_{t, \epsilon}|$ as a sum over these three intervals and recalling \eqref{eq-referee-1}, we obtain that
\begin{align*}
\E |\mathcal{P}_{t, \epsilon}|&\leq \sum_{k=0}^{c_1 \epsilon t}  \binom{nt + k}{k} \left(\frac{c^*\beta}{n}\right)^k + \sum_{k=c_1 \epsilon t}^{nt} \sum_{j=0}^{\epsilon t +1}    M_{n t, k + 1, j, \epsilon n } \left(\frac{c^*\beta}{n}\right)^k + \sum_{k> nt}  \binom{nt + k}{k} \left(\frac{c^*\beta}{n}\right)^k\\
&\leq  C_1 \left(1+\left(\frac{4 \mathrm{e}c^*\beta t}{c_1 \epsilon t}\right)^{c_1 \epsilon t}\right) + C_2\sum_{k=c_1 \epsilon t}^{nt} \sum_{j=0}^{\epsilon t +1} \frac{2^k }{k!} (nt)^{2j} (\epsilon n (k-j))^{k-2j} \left(\frac{c^*\beta}{n}\right)^k   +  o_n(1) \\
&\leq  C_1\left(1+\mathrm{e}^{c_2 \epsilon \log(1/\epsilon) t}\right) + C_2\sum_{k=c_1 \epsilon t}^{nt} \sum_{j=0}^{\epsilon t +1} \frac{(2c^* \beta \epsilon)^k k^{k}}{k!} \left(\frac{t}{\epsilon k}\right)^{2j}   +  o_n(1) \\
&\leq  C_1\left(1+\mathrm{e}^{c_2 \epsilon \log(1/\epsilon) t}\right) + C_2\sum_{k=c_1 \epsilon t}^{nt}(2 c^* \beta \mathrm{e} \epsilon)^k \mathrm{e}^{c_3 \epsilon \log(1/\epsilon) t}  +  o_n(1) \\
&\leq  C_1\left(1+\mathrm{e}^{c_2 \epsilon \log(1/\epsilon) t}\right) + C_2\mathrm{e}^{c_3 \epsilon \log(1/\epsilon) t}  +  o_n(1) \leq C\mathrm{e}^{C\sqrt{\epsilon}t}\,,
\end{align*}
where $o_n(1)\to_{n\to \infty} 0$  and $c_2, c_3, C_1, C_2, C$ are some large constants depending only on $\beta$.
\end{proof}

\section{A first look at the power law}\label{s:firstmoment}
In this section, we establish the following power law on the random metric generated by the continuous model.
\begin{prop}\label{prop-existence-exponent}
For all $\beta>0$, there exists $0<\theta<1$ such that
$$\E d_{(\delta, 1)} (0, 1) = \delta^{1 - \theta + o_\delta(1)}\,.$$
\end{prop}
In Subsection~\ref{subsec:exponent-0-1} we show that $\E d_{(1, n)}(0, n)$ grows polynomially in $n$ while the power is strictly between 0 and 1, and in Subsection~\ref{subsec:exponent}, we prove the existence of the exponent. To begin with we prove
the following self-similar property of the continuous model that will be applied repeatedly.
\begin{lemma}\label{lem-renormalization}
For all $\beta>0$ and all $x, \delta, \delta'>0$, we have
 \begin{equation}\label{eq-renormalization}
d_{(\delta \delta', \delta)}(0, x) \stackrel{\mathrm{law}}{=}\delta d_{(\delta', 1)}(0, x/\delta)\,.
 \end{equation}
\end{lemma}
\begin{proof}
By the scaling property of Poisson process and the definition of our continuous model, we can couple the long edges in metric $d_{(\delta \delta', \delta)}(\cdot, \cdot)$ and $d_{(\delta', 1)}$ in the manner such that for all $x, y\in \mathbb{R}$ there is an edge connecting $x$ and $y$ in $d_{(\delta \delta', \delta)}(\cdot, \cdot)$ if and only if there is an edge connecting $x/\delta$ and $y/\delta$ in $d_{(\delta', 1)}(\cdot, \cdot)$. Under this coupling, it is obvious that $d_{(\delta \delta', \delta)}(0, x) = \delta d_{(\delta', 1)}(0, x/\delta)$, completing the proof of the lemma.
\end{proof}

\subsection{Preliminary bound on the expected distance}\label{subsec:exponent-0-1}
In this subsection, we show that the expected distance in the LRP model grows neither linearly nor sub-polynomially.
\begin{prop}\label{prop-theta-0-1}
For $\beta>0$, there exists $C, c>0$ and $0<\alpha_1 < \alpha_2<1$ such that for all $0<\delta<1$
$$c \delta^{1-\alpha_2} \leq \E d_{(\delta, 1)} (0, 1) \leq C \delta^{1 - \alpha_1}\,.$$
\end{prop}
\begin{proof}[Proof of Proposition~\ref{prop-theta-0-1}: upper bound.]
For any $1/3\leq x\leq 1$, let $A_x$ be the event that there exists an edge of length in $(x/3, 1)$ joining intervals $[x/9, x/3]$ and  $[2x/3, 8x/9]$. On the event $A_x$, we denote by $\langle x_1, x_2\rangle$ this edge (pick an arbitrary one if there are multiple such edges). Therefore, for all $j\geq 1$ we have
\begin{align*}
\E d_{(9^{-j}, 1)}(0, x) &\leq \E (d_{(9^{-j}, x/3))} (0, x) \mid A_x) \P(A_x) + \E (d_{(9^{-j}, x/3)} (0, x)\mid A_x^c) \P(A_x^c)\\
&= \P(A_x) \E (d_{(9^{-j}, x/3)} (0, x_1) + d_{(9^{-j}, x/3)} (x_2, x)) +(1 - \P(A_x))\E d_{(9^{-j}, x/3)}(0, x)\\
& = \tfrac{x}{3}\P(A_x) \E (d_{(9^{-(j-1)}, 1)} (0, \tfrac{3x_1}{x}) + d_{(9^{-(j-1)}, 1)} (0,  \tfrac{3(x - x_2)}{x})) + \tfrac{x(1 - \P(A_x))}{3}\E d_{(9^{-(j-1)}, 1)}(0, 3)\,,
\end{align*}
where in the first equality we used the fact that $A_x$ is independent of $d_{(9^{-j}, x/3)}(0,x)$ since they involve edges on different length scales and in the last equality we have used Lemma~\ref{lem-renormalization}. Writing $d_j = \max_{1/3 \leq x\leq 1} \E d_{(9^{-j}, 1)}(0, x)$ and $\alpha = \min_{ 1/3\leq x\leq 1} \P(A_x)$ (indeed $\alpha \equiv \P(A_x)$), we deduce from the above inequality that
\begin{equation}\label{eq-theta-less-1}
d_j \leq (1-\alpha/3) d_{j-1}\,,
\end{equation}
where we have used the triangle inequality that $\E d_{(9^{-(j-1)}, 1)}(0, 3) \leq 3 \E d_{(9^{-(j-1)}, 1)}(0, 1)$.
Noting that $\alpha>0$ depends only on $\beta$, we apply \eqref{eq-theta-less-1} recursively and obtain that $\E  d_{(9^{-j}, 1)}(0, x)  \leq (1 - \alpha/3)^j$ for all $j\geq 1$ and $1/3\leq x\leq 1$. Combined with Lemma~\ref{lem-renormalization}, this yields the desired upper bound.
\end{proof}

Recall that $d^*_{(\delta, 1)}(x, y)$ is the distance between $x$ and $y$ using edges of length in $(\delta, 1)$ from the whole line instead of being confined to the interval of $(x, y)$ as in the case for $d_{(\delta, 1)}(x, y)$.
\begin{lemma}\label{lem-lower-bound-distance}
For all $\beta>0$, there exists a constant $C(\beta)>0$ such that for all $x>0$ and $0\leq u\leq 1$
$$\P(  d^*_{(\delta, 1)}(0, x) \leq 20 u \delta x) \leq u (C(\beta) \epsilon(\delta) )^{(x/2)\vee 1}\,,$$
where $\epsilon(\delta) \to 0$ as $\delta \to 0$.
\end{lemma}
\begin{proof}
As in the proof of Lemma~\ref{lem-exploration}, we consider the discretized model indexed with $n\in \mathbb{N}$ on $\mathbb{Z}/n$ where an edge is placed between $i/n$ and $j/n$ with probability $\beta (|i-j|)^{-2}$ if $|i-j|/n \in (\delta, 1)$. Clearly, the discretized model converges to $d_{(\delta, 1)}(\cdot, \cdot)$ as $n\to \infty$. Thus, it suffices to establish a corresponding  upper bound (uniform in $n$) on the discretized model.

Denote by $A_k$ the event that there is a path joining $0$ and $x$ in the discretized model of length at most $20u\delta x$ and $k$ hops. Then,
\begin{align}\label{eq-P-A-k}
\P(A_k) \leq   \tbinom{20u \delta x n + k-1}{k} \big(\tfrac{ c^*\beta}{\delta n}\big)^k \cdot 2^k\,,
\end{align}
where the binomial term counts the number of ways to obtain a sequence of gaps which sums to at most $20u \delta x n$, the factor $2^k$ amounts to the choices of directions to walk along for each gap, and the last term counts for the probability that there exist long edges started from specific points. Set $\epsilon = (\log (1/\delta))^{-1/4}$, and first consider the case that $x \leq \epsilon$, in which
a straightforward computation yields that
$$\limsup_{n\to \infty}\P(\exists k\in \mathbb{N}: A_k) \leq C \beta \epsilon u\,,$$
where $C>0$ is an absolute constant. Now, consider the case $x\geq \epsilon$. Using \eqref{eq-P-A-k} again, we obtain
$$\limsup_{n\to \infty}\P(\exists k\geq \epsilon^{-1}x/10: A_k)\leq \sum_{k\geq \epsilon^{-1}x/\10}\frac{(40 u c^*  x \beta)^k}{k!} \leq (C\beta\epsilon u)^{\epsilon^{-1} x/10}\,.$$
For $k\leq \epsilon^{-1} x/10$ and $\delta<1/100$ there are at least $(x/2) \vee 1$ edges with length at least $\epsilon$. Thus, in this case we have
$$\P(A_k)\leq \tbinom{20\delta u x n + k-1}{k} \big(\tfrac{2 c^*\beta}{\delta n}\big)^k (\tfrac{\delta}{\epsilon})^{(x/2)\vee 1}\,.$$
This then implies that
$$\limsup_{n\to \infty}\P(\exists k \leq \epsilon^{-1} x/10: A_k) \leq \sum_{1\leq k\leq \epsilon^{-1} x/10}\frac{(C\beta u x)^{k}}{k!} (\delta/\epsilon)^{(x/2)\vee 1} \leq u(\mathrm{e}^{2C \beta} \delta/\epsilon)^{(x/2) \vee 1}\,.$$
Altogether, this completes the proof of the lemma.
\end{proof}

We are now ready to complete the proof of Proposition~\ref{prop-theta-0-1}.
\begin{proof}[Proof of Proposition~\ref{prop-theta-0-1}: lower bound.]
It suffices to show that for any $\beta > 0$ and $\delta<\delta_0(\beta)$,
\begin{equation}\label{eq-recursion-lower-theta}
\E d_{(\delta^{j+1}, 1)}(0, 1) \geq 2 \delta \E d_{(\delta^{j}, 1)}(0, 1), \mbox{ for all }  j\in \mathbb{N}\,.
\end{equation}
In order to prove \eqref{eq-recursion-lower-theta}, we construct a short path joining $0$ and $1$ using edges of lengths in $(\delta^{j+1}, 1)$ in the following two steps: (1) construct a path $P$ joining $0$ and $1$ using edges of lengths in $(\delta^j, 1)$; (2) for each gap $(L_i, R_i)$ in $P$, find the geodesic between $L_i$ and $R_i$ using edges of lengths in $(\delta^{j+1}, \delta^j)$. Combine the path $P$ with all the geodesics joining the gaps yields a path $P^*$ joining $0$ and $1$ using edges of lengths in $(\delta^{j+1}, 1)$. Note that every path in $d_{(\delta^{j+1}, 1)}(\cdot, \cdot)$ can be constructed in this way. In order to bound the geodesic from below, it suffices to give a union bound on all such paths.  For each $P$, let $K$ be the number of hops. Denote by $G_i$ the length of the gaps (so $\sum_i G_i = \|P\|_{1}$), and denote by $G^*_i$ the length of the geodesic filling the gap $G_i$ using edges  of lengths in $(\delta^{j+1}, \delta^j)$. Conditioning on $G_i$, it is clear that $G^*_i \stackrel{law}{=} \delta^j  d^*_{(\delta, 1)}(0, \delta^{-j} G_i)$. Note that it suffices to consider the case when $P^*$ is a self-avoiding path. That is to say, the geodesics that fill in the gaps will have to use disjoint edges. With further conditioning on the values of $G_i$'s,  the BK inequality \cite{BK:85} allows us to consider the case where the $G_i^*$'s are independent. As for the independent case, we apply Lemma~\ref{lem-lower-bound-distance} and obtain that for a constant $C(\beta)>0$ and $\epsilon(\delta)$ (with $\epsilon(\delta)\to 0$ as $\delta \to 0$)
\begin{equation}\label{eq-BK-referee}
\P(\sum_i G_i^* \leq 16 u\delta t \mid K = k, \sum_i G_i = t) \leq u 2^k (C(\beta) \epsilon(\delta))^{t\delta^{-j}} \mbox{ for all } 0\leq  u\leq 1\,.
\end{equation}
Combined with Lemma~\ref{lem-exploration} and the fact that $P^*$ can be assumed to be self-avoiding without loss of generality, it follows that for all $0\leq u\leq 1$
\begin{equation}\label{eq-P-P*}
\P(\exists P \mbox{ such that } t\leq \|P\|_{1} \leq 2t \mbox{ and } \|P^*\|_{1} \leq 8\delta t) \leq u(C'(\beta) \epsilon'(\delta))^{t\delta^{-j}}\,,\end{equation}
where $C'(\beta)>0$ is another constant that depends only on $\beta$ and $\epsilon'(\delta)\to_{\delta \to 0} 0$. Applying \eqref{eq-P-P*}, we obtain that
\begin{align}
&\P(t\leq d_{(\delta^j, 1)} (0,1) \leq 2t, d_{(\delta^{j+1}, 1)}(0, 1) \leq  8\delta  d_{(\delta^j, 1)}(0, 1)) \nonumber\\
&\leq \sum_{i=0}^\infty  \P(\exists P \mbox{ such that } 2^i t\leq \|P\|_{1} \leq 2^{i+1}t \mbox{ and } \|P^*\|_{1} \leq 8\delta t) \nonumber \\
&\leq \sum_{i=0}^\infty 2^{-i+1} (C'(\beta) \epsilon'(\delta))^{t\delta^{-j}} \leq 4(C'(\beta) \epsilon'(\delta))^{t\delta^{-j}}\,. \label{eq-referee-page-8-1}
\end{align}
Therefore, we conclude that
\begin{align}
\E d_{(\delta^{j+1}, 1)}(0, 1) &\geq \E 8\delta  d_{(\delta^j, 1)}(0, 1) \one\{d_{(\delta^{j+1}, 1)}(0, 1) \geq 8\delta  d_{(\delta^j, 1)}(0, 1)\} \nonumber\\
&= \int_0^\infty 16\delta \P(t\leq d_{(\delta^j, 1)} (0,1) \leq 2t, d_{(\delta^{j+1}, 1)}(0, 1) \geq  8\delta  d_{(\delta^j, 1)}(0, 1)) dt \nonumber \\
&\geq  \int_0^\infty 16\delta (\P(t\leq d_{(\delta^j, 1)} (0,1) \leq 2t) - 4(C'(\beta) \epsilon'(\delta))^{t\delta^{-j}}) dt \nonumber \\
&= 8 \delta \E d_{(\delta^j, 1)}(0, 1) - 64\delta(\log(1/(C'(\beta) \epsilon'(\delta))))^{-1} \delta^j\,. \label{eq-referee-page-8-2}
\end{align}

Observe a trivial lower bound that $\E d_{(\delta^j, 1)} (0, 1) \geq \delta^j/2\beta$, where $\delta^j/2\beta$ is the expected distance one needs to travel away from the origin before seeing any point that is incident to an edge with length in $(\delta^j, 1)$. Therefore, we can choose $\delta_0(\beta)$ such that for all $\delta<\delta_0(\beta)$ the right hand side of the above display is at least $2\delta \E d_{(\delta^j, 1)}(0, 1)$. This completes the verification of \eqref{eq-recursion-lower-theta}, and thus completes the proof of the lower bound.
\end{proof}

The next corollary roughly states that $\E d_{(1, n)} (0, n)$ does not decrease with $n$.
\begin{cor}\label{cor-first-moment-does-not-decrease}
For any $\beta > 0$, there exists a constant $C_\beta > 0$ such that for all $0\leq s\leq n$
 $$\E d_{(1, n)} (0, s) \leq C_\beta \E d_{(1, n)} (0, n)\,.$$
\end{cor}
\begin{proof}
By \eqref{eq-recursion-lower-theta}, there exists a constant $C'_\beta > 0$ such that the desired inequality holds for all $s \leq n/C'_\beta$ with $C_\beta = 1/2$. Now consider $s\geq n/C'_\beta$. Let $\kappa = n\lfloor C'_\beta\rfloor /s$. Then using Lemma~\ref{lem-renormalization} and \eqref{eq-recursion-lower-theta} again, we get that
$$\E d_{(1, n)}(0, s) \leq \frac{1}{2} \E d_{(1, n\kappa)} (0, s\kappa) = \frac{1}{2} \E d_{(1, n\kappa)}(0, n\lfloor C'_\beta\rfloor) \leq \frac{1}{2} \E d_{(1, n)}(0, n\lfloor C'_\beta\rfloor) \leq \frac{\lfloor C'_\beta\rfloor}{2} \E d_{(1, n)} (0, n) \,.$$
Altogether, this completes the proof of the corollary.
\end{proof}

\subsection{Existence of the exponent}\label{subsec:exponent}
In this subsection we prove the existence of the exponent using a submutiplicative argument.

\begin{lemma}\label{lem-submultiplicative}
For any $\beta > 0$, there exists a constant $C_\beta > 0$ such that for all $0<\delta, \delta'\leq 1$
 $$\E d_{(\delta\delta', 1)} (0, 1) \leq C_\beta \E d_{(\delta, 1)} (0, 1)\E d_{(\delta', 1)} (0, 1)\,.$$
\end{lemma}
\begin{proof}
We need to construct a path $P$ joining $0$ and $1$. To this end, take the geodesic that achieves $d_{(\delta, 1)}(0, 1)$, and denote by $D = d_{(\delta, 1)}(0, 1)$ and by $K$ the number of hops in the geodesic. Then we can divide the gaps into $N\leq D /\delta + K$ segments whose lengths are at most $\delta$. We fill in each such segment with geodesics using edges of lengths in $(\delta \delta', \delta)$, thereby obtaining a path $P$ joining $0$ and $1$.
Applying Lemma~\ref{lem-exploration} (with suitable normalization), we obtain that
\begin{align*}
\E K & = \sum_{k=1}^\infty \P(K\geq k)\leq \sum_{k=1}^\infty \Big(\P(K\geq k,\frac{D}{\delta C} \leq k) + \P(K\geq k,\frac{D}{\delta C} > k) \Big)\\
&\leq  \sum_{k=1}^\infty \mathrm{e}^{-ck} + \frac{\E D}{\delta C}+1  \leq C_1 \frac{\E D}{\delta}\,,
\end{align*}
where $C$ and $c$ are constants in Lemma~\ref{lem-exploration} and $C_1$ is a constant depending on $C$ and $c$ (here we have used the fact that $\E D \geq c^\star\delta$ for a constant $c^\star > 0$).
Therefore, by Lemma~\ref{lem-renormalization} and Corollary~\ref{cor-first-moment-does-not-decrease}, we have
$$\E \|P\|_1 \leq C_\beta \delta \E N \E d_{(\delta', 1)}(0, 1) \leq C_\beta (C_1+1) \E d_{(\delta, 1)}(0, 1) \E d_{(\delta', 1)}(0, 1)\,,$$
for a constant $C_\beta>0$ as required.
\end{proof}

\begin{proof}[Proof of Proposition~\ref{prop-existence-exponent}] Define $\theta(\delta) = \log (\E d_{(\delta, 1)} (0, 1))/ \log \delta$. By Lemma~\ref{lem-submultiplicative}, we obtain that
$$\theta(\delta \delta') \leq \frac{\log \delta}{\log \delta + \log \delta'} \theta(\delta) + \frac{\log \delta'}{\log \delta + \log \delta'} \theta(\delta') +o(1) \,,$$
where $o(1)$ tends to 0 as $\delta\delta'\to 0$. This demonstrates  the subadditivity of $\theta(\delta)$ as $\delta \to 0$, thereby establishing $\theta(\delta)\to_{\delta \to 0} \theta$. By Proposition~\ref{prop-theta-0-1}, we see that $0<\theta<1$.
\end{proof}

\section{A second look at the power law}\label{s:SecondMoment}
In this section, we strengthen the power law obtained in the previous section by sharpening the estimate up to a multiplicative constant. We consider a fixed $\beta>0$ and write $\theta = \theta(\beta)$ as in Proposition~\ref{prop-existence-exponent}. For any $\gamma>0$ and $\theta' < \theta$, define $M^*$ to be $(\gamma, \theta')$-good if
\begin{equation}\label{eq-def-theta-good}
\E d_{(1, M)}(0, M) \leq \gamma (M/M^*)^{\theta'} \E d_{(1, M^*)}(0, M^*) \mbox{ for all } 1\leq M\leq M^*\,.
\end{equation}
\subsection{A bound on the second moment at good points}
In this subsection, we prove the next lemma which states that the second moment of $d_{(1, M^*)}(0, M^*)$ is controlled by the first if $M^*$ is good.
\begin{lemma}\label{lem-second-moment-good}
For all $\beta>0$, there exist constants $0<\theta'<\theta$ and $\gamma, C >0$ such that for all $M^*$ that is $(\gamma, \theta')$-good,
$$\E (d_{(1, M^*)}(0, M^*))^2 \leq C (\E d_{(1, M^*)}(0, M^*))^2\,.$$
\end{lemma}

\begin{proof}
Throughout the proof, we denote by $C_1, C_2, c_1, c_2, \ldots$ positive constants which depend only on $\beta$. 
Write $M^* = m^n$ for a suitable $m\geq 10$ to be selected. For $j= 0, \ldots, n-1$, we wish to relate the moments for the LRP distances using edges of lengths in $(1, m^j)$ to that using edges of lengths in $(1, m^{j+1})$. To this end, consider $8 m^j\leq \Upsilon \leq m^{j+1}$ and we construct the following short path as an upper bound on $d_{(1, m^{j+1})} (0, \Upsilon)$: (1) Let $0\leq x_1\leq x_2 \leq \ldots \leq x_K \leq \Upsilon$ be the sequence of points which are incident to an edge of length in $(m^{j}, m^{j+1})$. (2) Interpolate into the sequence of $(x_i)$ a minimal number of points to obtain a sequence of $0\leq y_1 \leq y_2 \leq \ldots \leq y_{K'} \leq \Upsilon$ such that $|y_{i+1} - y_i| \leq m^j$ for all $i<K'$. (3) Let $A_y = \{i<K': \mbox{ there is an edge of length in } (m^j, m^{j+1}) \mbox{ jumping over } (y_i, y_{i+1})\}$ (by ``jumping over'' $(y_i, y_{i+1})$ we mean that there is an edge connecting some $0\leq z \leq y_i$ and $y_{i+1} \leq z' \leq \Upsilon$) and let $\tau = \arg\max d_{(1, m^j)}(y_i, y_{i+1})$ (in case of multiple maximizers, pick an arbitrary one). Denote by $\langle y_{i_L}, y_{i_R}\rangle$ the long edge that jumps over $(y_{\tau}, y_{\tau+1})$). If $\tau\in A_y$, include edge  $\langle y_{i_L}, y_{i_R}\rangle$ in the path and complete the path by filling the gaps in $\{(y_i, y_{i+1}): i< i_L \mbox{ or } i\geq i_R\}$ using the geodesic in $d_{(1, m^j)}(y_i, y_{i+1})$. Otherwise use the geodesics in $d_{(1, m^j)}(y_i, y_{i+1})$ to fill in all the gaps and obtain a path. Due to the construction, it is clear that
\begin{equation}\label{eq-crude-upper-second}
d_{(1, m^{j+1})} (1, \Upsilon) \leq \max_{i\notin A_y} d_{(1, m^{j})}(y_{i}, y_{i+1})  + \sum_{i\neq \tau} d_{(1, m^j)}(y_i, y_{i+1})\,.
\end{equation}
It is crucial to estimate $\E |[K'] \setminus A_y|$. To this end, we first consider  $$A_x = \{i<K: \mbox{ there is an edge of length in } (m^j, m^{j+1}) \mbox{ jumping over } (x_i, x_{i+1})\}\,.$$ Write $\mathcal{X} = \{x_i: i\in [K]\}$. For $0\leq s\leq t\leq \Upsilon$, denote by $E_{s, t}$ the event no edge of length in $(m^j, m^{j+1})$ joins $(0, s)$ and $(t, \Upsilon)$.  Furthermore, denote by $\mathcal X^{(2)}$ the collection of pairs $( z_1, z_2) \in \mathcal X^2$ such that there exists no edge of length in $(m^j, m^{j+1})$ connecting either $z_1$ or $z_2$ to some point $z$ with $z_1\leq z\leq z_2$. Define $f(s, t)$ such that for all Borel subsets $S$ and $T$,
$$\int_{s\in S, t\in T} f(s, t) dsdt = \E |\mathcal X^{(2)} \cap S\times T|\,.$$
For a justification of the existence of the density $f(s, t)$ in the preceding inequality,  we see that $\E|\mathcal X^2 \cap \cdot|$ is a measure on $\R^2$ which is dominated by the measure $\E |\mathcal X \cap \cdot| \otimes \E |\mathcal X \cap \cdot|$. Since the latter is absolutely continuous with respect to Lebesgue measure with density bounded by $(\frac{2\beta}{m^j})^2$, the density for the former exists and $f(s, t) \leq (\tfrac{2\beta}{m^j})^2$. Thus, we obtain
\begin{align}\label{eq-K-A-x}
\E |[K] \setminus A_x| = \int_{0\leq s\leq t\leq \Upsilon}  f(s, t) \P(\mathcal{X} \cap (s, t) = \emptyset) \P(E_{s, t}) dt ds\,.
\end{align}
In order to verify the preceding equality, we see that 1) conditioning on $\langle s, t\rangle \in \mathcal X^{(2)}$ will not change the probability for the event $\mathcal X \cap (s, t) = \emptyset$, since with probability 1 there exists no edge connecting either $s$ or $t$ to some $z\in (s, t)$; 2) the event $E_{s, t}$ is independent of the aforementioned two events.  Now, write $\mathcal X_{\leq}$ as the collection of points $0\leq z\leq \Upsilon$ for which there exists $0\leq z' \leq z - m^j$ such that the edge connecting $z'$ and $z$ occurs, and respectively write $\mathcal X_{\geq}$ as the collection of points $0\leq z\leq \Upsilon$ for which there exists $z+m^j\leq z' \leq \Upsilon$ such that the edge connecting $z'$ and $z$ occurs. Clearly we have $\mathcal X_{\leq}\cup \mathcal {X}_{\geq} \subseteq \mathcal X$. The advantage of considering $\mathcal X_\leq$ and $\mathcal X_\geq$ is that they are both Poisson point process while $\mathcal X$ is not due to parity issues.  When $\frac{s+t}{2} \leq\Upsilon$ and $s\geq m^j$, we have
\begin{align*}
\P(\mathcal X \cap (s, t) = \emptyset)& \leq \P(\mathcal X_{\geq} \cap (s, t) = \emptyset) = \mathrm{e}^{-\E |\mathcal X_{\geq} \cap (s, t)|}\\
&\leq \mathrm{e}^{-\frac{t-s}{2}\int_{m^j}^{\Upsilon/2} \frac{\beta}{z^2} dz} \leq \mathrm{e}^{-\frac{\beta(t-s)}{4 m^j}}\,.
\end{align*}
A similar estimate holds for the case of $\frac{t+s}{2} \geq \Upsilon$ and $t\leq \Upsilon - m^j$ using $\mathcal X_{\leq}$. In addition, we have
\begin{align*}
\P(E_{s, t}) = \mathrm{e}^{-\int_0^s\int_{t \vee (z+ m^j)}^{\Upsilon} \tfrac{\beta}{(w-z)^2} dwdz} =
\begin{cases} \Big(\frac{\Upsilon (t-s)} {(\Upsilon -s)t}\Big)^\beta, & \mbox{ if } t-s\geq m^j\,,\\
\Big(\frac{\Upsilon m^j} {(\Upsilon -s)t}\Big)^\beta \mathrm{e}^{-\beta \frac{s+m^j-t}{m^j}}, & \mbox{ if } |t-s| < m^j\,.
\end{cases}
\end{align*}
Therefore, we always have
$$\P(E_{s, t}) \leq \Big(\frac{\Upsilon ((t-s)\vee m^j)} {(\Upsilon -s)t}\Big)^\beta\,.$$  
Plugging the above estimates into \eqref{eq-K-A-x}, we obtain that
\begin{align*}
\E |[K] \setminus A_x|
& \leq \int_{0}^{\Upsilon}\int_{s}^{\Upsilon} (\tfrac{2\beta}{m^j})^2 \mathrm{e}^{- \beta \frac{t-s}{4 m^j}}\big((\tfrac{\Upsilon ((t-s)\vee m^j)} {(\Upsilon -s)t}\big)^\beta dtds  \leq  C_1 (m^{1-\beta} \vee \log m )\,.
\end{align*}
Here the last inequality follows from a straightforward computation, and one should note that the main contribution to the integration comes from the region when $t-s$ is about the order of $m^j$ (due to the exponential term $\mathrm{e}^{-\beta \frac{t-s}{4m^j}}$).
We now bound the number of interpolations added to obtain the sequence $\{y_i\}$. To this end, define
$$B = \{x \in (0, \Upsilon): \mbox{ there is no edge in } (m^{j}, m^{j+1}) \mbox{ jumping over } x\}\,.$$ Denoting by $\mathcal{L}(S)$ the Lesbegue measure of $S$ for $S\subseteq \mathbb{R}$, we have that
\begin{align}
\E \mathcal{L}(B) &= \int_0^{\Upsilon} \P(x\in B) dx \leq \int_0^{\Upsilon} \mathrm{e}^{-\int_{m^j}^{x \vee (\Upsilon-x)} \beta s^{-2} \min\{s, x, \Upsilon - x\}ds} dx \nonumber\\
&\leq  2 m^j + 2 \int_{m^j}^{\Upsilon/2} \mathrm{e}^{-\int_{m^j}^{x} \beta s^{-1}  ds } dx  \leq C_2 (m^{1-\beta} \vee \log m ) m^j\,.
\end{align}
When $\beta<1$ the reverse of this inequality also holds. In particular, denoting by
$$B' = \{x \in (0, m): \mbox{ there is no edge in } (1,m) \mbox{ jumping over } x\}\,,$$
we have $\E \mathcal{L}(B') \geq c_1 m^{1-\beta}.$  It follows that
\begin{equation}\label{eq-referee-beta-less-1}
\E d_{(1,m)}(0,n) \geq c_1 m^{1-\beta} \mbox{ and hence }\theta \leq\beta\,,
\end{equation}
where $\theta$ is as in Proposition~\ref{prop-existence-exponent}.
 Note that for every point in $[K']\setminus A_y$, it should either be in $[K]\setminus A_x$ or come from interpolating a long interval in $B$. 
Therefore, we obtain that
$$\E | [K'] \setminus A_y| \leq \E |[K]\setminus A_x| + m^{-j} \E \mathcal{L} (B) \leq C_3 (m^{1-\beta} \vee \log m )\,.$$
 Combined with \eqref{eq-crude-upper-second}, it follows that
\begin{equation*}
\E (d_{(1, m^{j+1})}(0, \Upsilon))^2 \leq C_4((m^{1-\beta} \vee \log m ) \max_{s\leq m^j}\E (d_{(1, m^{j})}(0, s))^2  + m^2  \E (\max_{i\neq \tau} d_{(1, m^j)}(y_i, y_{i+1}))^2)\,,\end{equation*}
where we absorb the term $2\E(\max_{i\notin A_y} d_{(1, m^{j})}(y_{i}, y_{i+1})  \sum_{i\neq \tau} d_{(1, m^j)}(y_i, y_{i+1}))$ by its upper bound $\E((\max_{i\notin A_y} d_{(1, m^{j})}(y_{i}, y_{i+1}) )^2 + \E( \sum_{i\neq \tau} d_{(1, m^j)}(y_i, y_{i+1})))^2$.
Write $\Gamma_j = \max_{s\leq m^j} \E (d_{(1, m^j)} (0, s))^2$, we obtain that (as $\Upsilon$ is arbitrary number in $[8m^j, m^{j+1}]$)
\begin{equation} \label{eq-less-crude-second}
\Gamma_{j+1} \leq  C_4((m^{1-\beta} \vee \log m ) \Gamma_j  + m^2  \E (\max_{i\neq \tau} d_{(1, m^j)}(y_i, y_{i+1}))^2).
\end{equation}

Since $d_{(1, m^j)}(y_i, y_{i+1})$ are independent for all $i\in [K']$ and is independent of $\{y_i\}$, we condition on $K' = k'$ and $\{y_{i+1} - y_i = s_i\}$ and compute
\begin{align}\label{eq-to-integrate}
&\E  (\max_{i\neq \tau} ( d_{(1, m^j)}(y_i, y_{i+1}))^2 \mid K' = k', y_{i+1} - y_i = s_i) \nonumber\\
 =& \int_0^\infty z \P(\max_{i\neq \tau} d_{(1, m^j)}(0, s_i) \geq z) dz \nonumber \\
\leq&  \sum_{i_1, i_2 \in [k']}\int_0^\infty z \P( d_{(1, m^j)}(0, s_{i_1}) \geq z) \P( d_{(1, m^j)}(0, s_{i_2}) \geq z)dz \nonumber \\
 \leq& (k')^2 \max_{i_1, i_2\in [k']} \E (d_{(1, m^j)}(0, s_{i_1})) \int_0^\infty  \P( d_{(1, m^j)}(0, s_{i_2}) \geq z) dz \nonumber \\
 =&  (k')^2 \max_{s\in [0, m^j]} (\E d_{(1, m^j)}(0, s))^2\,,
\end{align}
where the third inequality follows from Markov's inequality and the last inequality we used the fact that $s_i \leq m^j$ for all $i\in [k']$.
Averaging both sides of \eqref{eq-to-integrate} over the conditioning and using Corollary~\ref{cor-first-moment-does-not-decrease} and the fact that $\E [K']^2\leq C_5 m^2 \beta^2$, we obtain that
$$\E  (\max_{i\neq \tau} ( d_{(1, m^j)}(y_i, y_{i+1}))^2 \leq  C_6 m^2 (\E d_{(1, m^j)}(0, m^j))^2\,.$$
Plugging it into \eqref{eq-less-crude-second}, we obtain that
\begin{equation*}
\Gamma_{j+1} \leq   C_7 ((m^{1-\beta} \vee \log m )\Gamma_j  + m^4  (\E d_{(1, m^j)}(0, m^j))^2)\,.\end{equation*}
Writing $\eta = C_7(m^{1-\beta} \vee \log m )$ and applying the above recursively for $j=0, \ldots, n-1$, we obtain
\begin{equation}\label{eq-recursion-second-moment}
\Gamma_n \leq C_8 m^4 \sum_{j=1}^{n-1} \eta^{n-j}    (\E d_{(1, m^j)}(0, m^j))^2\,.
\end{equation}
Using \eqref{eq-referee-beta-less-1} and the fact that $\theta>0$ in general, we can choose $m$ and $\theta'$ depending only on $\beta$ such that $\eta m^{-2\theta'} \leq 1/2$. Combined with \eqref{eq-recursion-lower-theta} and the fact that $M^*$ is $(\gamma, \theta')$-good (recall \eqref{eq-def-theta-good}), it follows that
$$\Gamma_n \leq  C_8 \gamma m^4 \sum_{j=1}^{n-1} \eta^{n-j} m^{-2\theta'(n-j)}  (\E d_{(1, m^n)}(0, m^n))^2 \leq2 C_8 \gamma m^4 (\E d_{(1, m^n)}(0, m^n))^2\,.$$
Since $m$ and $\gamma$ depend only on $\beta$, this completes the proof.
\end{proof}

\subsection{Lower bound on distances at good points}
In this subsection, we study the lower bounds on distances at good points.
For $0<\iota<1/3$, define the $\iota$-segment-to-segment distance
\begin{equation}\label{eq-def-segment-to-segment}
d^\iota_{(1, M)} (0, x)= \min_{y\in [0, \iota x], z\in [(1-\iota)x, x]} d_{(1, M)}(y, z)\,,
\end{equation}
to be the length of the geodesic joining  $[0, \iota x]$ and $[(1-\iota)x, x]$. We now show that the expected segment-to-segment distance has the same order of magnitude as the original point-to-point distance.
\begin{lemma}\label{lem-segment-point}
For all $\beta>0$, $0<\theta'<\theta$ and $\gamma>0$, there exists  a constant $\iota>0$ such that if $M^*$ is $(\gamma, \theta')$-good, we have
\begin{equation}\label{e:lem-segment-point1}
\E d^\iota_{(1, M^*)} (0, M^*)  \geq \tfrac{1}{2} \E d_{(1, M^*)} (0, M^*)\,.
\end{equation}
Therefore, for a constant $c>0$ that depends only on $\beta$ we have
\begin{equation}\label{e:lem-segment-point2}
\P(d^\iota_{(1, M^*)} (0, M^*)  \geq \tfrac{1}{4} \E d_{(1, M^*)} (0, M^*))\geq c\,.
\end{equation}
\end{lemma}
\begin{proof}
Let $\tau\in [0, \iota M^*]$ and $\tau' \in [(1-\iota)M^*, M^*]$ be such that $d^\iota_{(1, M^*)} (0, M^*) = d_{(1, M^*)} (\tau, \tau')$.
By the triangle inequality, we see that
\begin{equation}\label{eq-triangle-segment}
d_{(1, M^*)} (0, M^*) \leq  d_{(1, M^*)} (\tau, \tau') +  d_{(1, M^*)} (0, \tau) + d_{(1, M^*)} (\tau', M^*)\,.
\end{equation}
A key observation is that the triple $\{(\tau, \tau', d_{(1, M^*)} (\tau, \tau'))\}$ is measurable with respect to the $\sigma$-field generated by random edges $(u, v)$ where $\{u, v\} \not\subset (0, \iota M^*)$ and $\{u, v\}\not\subset ((1-\iota)M^*, M^*)$. In particular, we have that $\{(\tau, \tau', d_{(1, M^*)} (\tau, \tau'))\}$ is independent of $d_{(1, M^*)}(0, \tau)$ and $d_{(1, M^*)}(\tau', M^*)$. Combined with \eqref{eq-triangle-segment}, it follows that
\begin{align*}\E d_{(1, M^*)} (0, M^*)& \leq \E d_{(1, M^*)} (\tau, \tau') + 2\max_{x\in (0, \iota M^*)}\E d_{(1, M^*)} (0, x)\\
 &\leq \E d_{(1, M^*)} (\tau, \tau') + 2\gamma \iota^{-\theta'}\E d_{(1, M^*)} (0, M^*)\,,
\end{align*}
where the second inequality follows from the definition of $(\gamma, \theta')$-good. Choosing $\iota$ such that $\iota^{-\theta'} \gamma < 1/4$ completes the proof of~\eqref{e:lem-segment-point1}.  The lower bound follows by the second moment method since by Lemma~\ref{lem-second-moment-good},
\begin{equation*}
\E [d^\iota_{(1, M^*)} (0, M^*)]^2 \leq \E [d_{(1, M^*)} (0, M^*)]^2 \leq C [\E d_{(1, M^*)} (0, M^*)]^2 \leq 4 C [\E d^\iota_{(1, M^*)} (0, M^*)]^2. \qedhere
\end{equation*}
\end{proof}

We say $M^*$ is $(r, \gamma, \theta')$-supergood if for all $j=1, \ldots, \log \log_r M^*$ there exists $M^* \mathrm{e}^{-r^{j}}\leq M_j\leq  M^* \mathrm{e}^{-r^{j-1}}$ such that $M_j$ is $(\gamma, \theta')$-good.
\begin{lemma}\label{lem-super-good}
For all $\beta>0$ and $0<\theta'<\theta^*<\theta$, there exists $r>1$ such that if $M^*$ is $(1, \theta^*)$-good then $M^*$ is $(r, 1, \theta')$-supergood.
\end{lemma}
\begin{proof}
For $x \geq 0$, define $\phi(x) = \E d_{(1, \mathrm{e}^x)}(0, \mathrm{e}^x)$ and $\psi(x) = \log \phi(x) - \theta' x$. It is clear that $\mathrm{e}^z$ is $(1,\theta')$-good if and only if $z$ is a record value for $\psi(x)$ (that is to say, $\psi(x)\leq \psi(z)$ for all $x\leq z$). Write $m^* = \log M^*$. Since $M^*$ is $(1, \theta^*)$-good, we have that for $x\leq m^*$
\begin{equation}\label{eq-psi-upper}
\psi(x) - \psi(m^*) = \log (\phi(x)/\phi(m^*)) - \theta'(x- m^*) \leq (\theta^* - \theta')(x-m^*)\,.\end{equation}
By Lemma~\ref{lem-renormalization}, we see that $\E d_{(1, \mathrm{e}^{m^*})}(0, \mathrm{e}^{m^*}) \leq \mathrm{e}^{m^* -x} \E d_{(1, \mathrm{e}^x)}(0, \mathrm{e}^x)$ or equivalently  $\phi(m^*) \leq \mathrm{e}^{m^* -x} \phi(x)$, and thus
\begin{equation}\label{eq-psi-lower}
\psi(x) - \psi(m^*) \geq (1 - \theta') (x - m^*)\,.
\end{equation}
For $k=0,1, \ldots$, we let
$$z_k = \min\{z: \psi(z) - \psi(m^*) \geq -m^*(\theta^* - \theta')\big(\tfrac{\theta^* - \theta'}{1- \theta'}\big)^k\}\,.$$
By \eqref{eq-psi-upper} and \eqref{eq-psi-lower}, we could deduce that
$$m^* \big(1 - \big(\tfrac{\theta^* - \theta'}{1- \theta'}\big)^{k}\big) \leq z_k \leq m^* \big(1 - \big(\tfrac{\theta^* - \theta'}{1- \theta'}\big)^{k+1}\big)\,.$$
Clearly $\{z_k\}$ are record values of $\psi(x) - \psi(m^*)$ and thus record values of $\psi(x)$. Hence, $\{\mathrm{e}^{z_k}\}$ are $\theta'$-good. Setting $r = (1 - \theta')/(\theta^* - \theta')$, we complete the proof.
\end{proof}

We record the following consequence of Lemmas~\ref{lem-renormalization} and \ref{lem-super-good} for future reference.
\begin{cor}\label{cor-referee}
For $\theta'<\theta$ and $r>1$, if $M^*$ is $(r, 1, \theta')$-supergood, then there exists $\kappa_0 = \kappa_0(\theta', \beta, r)$ such that $\kappa M^*$ is $(r, \gamma_{\kappa_0}, \theta')$-supergood for all $\kappa\geq \kappa_0$.
\end{cor}

Analogous to the definition of $ d^*_{(1, M^*)}(x, y)$, we define  $d^{*\iota}_{(1, M^*)}(x, y)$ be the segment-to-segment distance between $x$ and $y$ which allows to use edges in the whole line.
\begin{lemma}\label{lem-distance-tight}
For $\beta>0$, let $\theta'>0$ be specified as in the statement of Lemma~\ref{lem-second-moment-good}. For $r, \gamma\geq 1$, the following hold uniformly for all $(r, \gamma, \theta')$-supergood $M^*$. For any $\epsilon>0$, there exists $C>0$ so that
\begin{equation}\label{e:lem-distance-tight1}
\P(C^{-1}\E d_{(1, M^*)}(0, M^*)\leq d_{(1, M^*)}(0, M^*) \leq C \E d_{(1, M^*)}(0, M^*)) \geq 1-\epsilon\,.
\end{equation}
In addition, for any fixed $\kappa_0<1$, there exists a constant $c_{\kappa_0, \epsilon}$ such that
\begin{equation}\label{eq-lower-bar-segment}
\P(  d^*_{(1, M^*)}(0, \kappa M^*)\geq  c_{\kappa_0, \epsilon} \E d_{(1, M^*)}(0, M^*)) \geq 1-\epsilon \mbox{ for all } \kappa\geq \kappa_0\,.
\end{equation}
Furthermore, there exist $\iota, c>0$ such that
$$\P(  d^{*\iota}_{(1, M^*)}(0, \kappa M^*)\geq  c_{\kappa_0, \iota} \E d_{(1, M^*)}(0, M^*)) \geq c \mbox{ for all } \kappa\geq \kappa_0\,.$$
\end{lemma}
\begin{proof}
The upper bound on $d_{(1, M^*)}(0, M^*)$ follows from Markov's inequality. For the lower bound, we give a proof for $ d^*_{(1, M^*)}(\cdot, \cdot)$, and the proof on the segment-to-segment distance will follow in the same manner (note that for segment-to-segment distance, we only show a result of with positive chance).   Let $\iota>0$ be a number sufficiently small to satisfy the assumption in the statement of Lemma~\ref{lem-segment-point}. Therefore, by Lemma~\ref{lem-segment-point}, we get for all $\theta'$-good $M$
\begin{equation}\label{eq-lower-segment}
\P(d^{\iota}_{(1, M)}(0, M) \geq  \tfrac{1}{4}\E d_{(1, M)}(0, M)) \geq c'\,,\end{equation}
where $c'>0$ is a constant (uniformly for all $\theta'$-good $M$). Since $M^*$ is super good, we can then choose $M_1, \ldots, M_K$ such that all $M_i$'s are $\theta'$-good, and
$$M_{i+1}/M_i\geq R\,, \,M_1 \geq M^*/L\,, \mbox{ and } \mbox{$\sum_{i=1}^K$} M_i \leq \kappa M^*\,,$$ where $R, L$ are sufficiently large numbers to be fixed (independent of $M^*$). Write $x_j = \sum_{i=1}^j M_j$ and $x_0 = 0$. Let $A_j$ be the event that there exists a long edge joining $(-x_{j-1}, x_{j-1})$ and $(x_{j-1} + \iota M_j, M^*)$ or $(-M^*, -x_{j-1} - \iota M_j)$, and $B_j$ be the event that there exists a long edge joining $(-x_{j-1} - (1- \iota)M_j, x_{j-1} + (1-\iota)M_j)$ and $(x_j, M^*)$ or joining $(-M^*, -x_j)$ and $(-x_{j-1} - (1- \iota)M_j, x_{j-1} + (1- \iota)M_j)$. First of all, we see that
$$\sum_{j=1}^K\P(A_j) \leq \sum_{j=1}^K\frac{2\beta x_{j-1}}{\iota M_j} \leq \epsilon/3\,,$$
if we select $R$ large enough. Let $A = \cup_{j=1}^K A_j$. When $A^c$ and $B_j^c$ hold, for a path to escape $(-x_{j-1},x_{j-1})$ it must perform a segment to segment crossing of $(x_{j-1}, x_{j})$ or $(-x_{j}, -x_{j-1})$ and hence we see that
$$ d^*_{(1, M^*)}(0, M^*) \geq \one_{A^c} \sum_{j=2}^K \one_{B_j^c} \min\{ d^{\iota}_{(1, M^*)}(x_{j-1}, x_{j}), d^{\iota}_{(1, M^*)}(-x_{j}, -x_{j-1})\}\,.$$
Observe that on $A^c$, the events $B_j$'s are independent of each other, and the random variables $d^{\iota}_{(1, M^*)}(x_{j-1}, x_{j}), d^{\iota}_{(1, M^*)}(-x_{j}, -x_{j-1})$ for $j\geq 1$ are independent of each other, and from $B_j$'s. Combined with \eqref{eq-lower-segment}, it follows that
$$\P(\one_{B_j^c} d^{\iota}_{(1, M^*)}(x_{j-1}, x_{j}) \geq \tfrac{1}{4} \E d_{(1, M^*)}(0, M_j)) \geq c' \P(B_j^c) \geq c^*,$$
where $c^*>0$ is a constant. Furthermore, we have $\P(B_j \mid A^c) = \P(B_j \mid A_j^c)$, which is bounded away from 0. So if we choose $K$ sufficiently large, we obtain that with probability at least $1-\epsilon$,
$$ d^*_{(1, M^*)}(0, M^*) \geq \tfrac{1}{4} \min_{j\in [K]} \E d_{(1, M^*)}(0, M_j) \,.$$
Recalling that $M_i\geq M^*/L$, we see that  $\min_{j\in K} \E d_{(1, M^*)}(0, M_j) \geq \E d_{(1, M^*)}(0, M^*)/L$ (by the triangle inequality), completing the proof of \eqref{e:lem-distance-tight1}.  The rest of the lemma follows similarly.
\end{proof}
\subsection{Super-multiplicativity}
In this subsection, we prove a weak version of super-multiplicativity for distances in long range percolation.
\begin{lemma}\label{lem-super-multiplicative}
For all $\beta>0$, there exists $0<\theta^*<\theta$ such that if $M^*$ is $(1,\theta^*)$-good, we have
$$\E d_{(1, M^* M)} (0, M^* M) \geq c \E d_{(1, M^*)}(0, M^*) \cdot \E d_{(1, M)}(0, M)\,,$$
for all $M>0$ and a constant $c = c(\beta)>0$.
\end{lemma}
\begin{proof}
Let $\iota$ be a sufficiently small constant as in the statement of Lemma~\ref{lem-distance-tight}, and suppose $\theta^*$ is sufficiently close to $\theta$. Then, an application of Lemmas~\ref{lem-super-good} and \ref{lem-distance-tight} as well as Corollary~\ref{cor-referee} gives that \begin{equation}\label{eq-point-whole-line-lower}
\P(  d^*_{(1,  M^*)}(0, \kappa M^*) \geq \epsilon \cdot c_{\kappa_0} \E d_{(1, M^*)}(0, M^*))  \geq g(\epsilon)\,, \mbox{ for all } \kappa\geq \kappa_0\,,
\end{equation}
where $c_{\kappa_0}$ is a constant that depends on $\kappa_0$ and $g(\epsilon)\to 1$ as $\epsilon\to 0$.
Both the result and the proof of the current lemma are strengthened versions of those of Proposition~\ref{prop-theta-0-1}. Here \eqref{eq-point-whole-line-lower} plays a crucial role in the improvement. Following the proof of Proposition~\ref{prop-theta-0-1}, we construct short paths joining $0$ and $MM^*$ using edges of lengths in $(1, MM^*)$ in the following two steps: (1) construct a path $P$ joining $0$ and $MM^*$ using edges in $(M^*, MM^*)$; (2) for each gap $(L_i, R_i)$ in $P$, find the geodesic between $L_i$ and $R_i$ using edges in $(1, M^*)$. Combine the path $P$ with all the geodesics joining the gaps yields a path $P^*$ joining $0$ and $MM^*$ using edges of lengths in $(1, MM^*)$. Note that the geodesic in $d_{(1, MM^*)}(\cdot, \cdot)$ can be constructed in this way, and that it suffices to study those $P^*$'s which use each edge at most once (thus $P^*$ satisfies this property in all its appearances). In order to bound the distance from below, it suffices to give a union bound on all such paths.

For each $P$, denote by $\{G_i: i=1, \ldots, N\}$ the length of the gaps (so $\sum_{i=1}^N G_i = \|P\|_{1}$). For $\kappa>0$, let $N_\kappa = N_{\kappa,P} = |\{i\in N: G_i \geq \kappa M^*\}|$. Write $$\tilde{\mathcal{P}}_{t, \kappa} = \{P \in \tilde{\mathcal{P}}_t: N_\kappa \leq \kappa t/M^*\}\,,$$ where $\tilde{\mathcal{P}}_t$ is the collection of paths $P$ starting from origin using edges of lengths in $(M^*, M^*M)$ such that $\|P\|_1 \in(t/2, t)$ (Note that the tilde version of $\mathcal{P}_{\cdot, \cdot}$ is the same as the non-tilde version defined in Section~\ref{sec:explore} up to scaling). Then, after scaling by $M^*$ Lemma~\ref{lem-exploration} gives that
\begin{equation}\label{eq-P-t-kappa}
\E |\tilde{\mathcal{P}}_{t, \kappa}| \leq \mathrm{e}^{C_1 \sqrt{\kappa} t /M^*}
 \end{equation}
for a constant $C_1>0$. Write $\tilde{\mathcal{P}}_t^C = \{P \in \tilde{\mathcal{P}}_t: N\geq Ct/M^*\}$. Scaling by $M^*$ and applying Lemma~\ref{lem-exploration} again, we can choose $C_2>0$ such that
\begin{equation}\label{eq-P-t-C}
\P(\tilde{\mathcal{P}}_t^{C_2} \neq\emptyset) \leq \mathrm{e}^{-c_2 t/M^*}\mbox{ for a constant }c_2>0\,.
\end{equation} Now, fix  a path $P \in \tilde{\mathcal{P}}_{t, \kappa} \setminus \tilde{\mathcal{P}}_t^{C_2}$.
We see that we can select intervals $I_i = (x_i, y_i)$ which are placed in the gaps of path $P$ of length $M^\star = M^*/2(C_2\vee 1)$ for $i=1, \ldots, K = t/16M^*$ such that the minimal distance between these intervals in the same gap is at least $M^*$.  However, if the gaps overlap, intervals from different gaps may overlap. By Corollary~\ref{cor-referee} we see that $M^\star$ is $(\gamma, \theta^*)$-good for a $\gamma = \gamma(C_2, \beta)>0$. Choose $\iota>0$ to satisfy the assumption in Lemma~\ref{lem-segment-point}. Define $\mathcal J$ to be collection of subsets $J\subseteq [K]$ with the following property: there exist disjoint edges $\{e_j: j\in J\}$ of length in $(1, M^*)$ such that $e_j$ jumps over the first $\iota$ fraction or the last $\iota$ fraction of interval $I_j$ for every $j\in J$. Then, we have
$$\|P^*\|_{1} \geq \min_{J\in \mathcal J}\sum_{j\not\in J}  d^{*\iota}_{(1, M^*)} (x_i, y_i)\,.$$
Let $B_i$ be the event that there exists an edge of length in $(1, M^*)$ jumping over $I_i$, we have $\P(B_i) \leq c_{\beta, \iota, C_2} < 1$ for all $i$. Applying BK's inequality, we obtain that
\begin{equation}\label{eq-referee-2}
\P(J \in \mathcal J) \leq c_{\beta, \iota, C_2}^{|J|}\,, \mbox{ for all } J\subseteq [K]\,.
\end{equation}
Since the intervals $I_i$ are well separated, the events $B_i$ from the same gap are independent of each other. Conditioning on $B_i$, we see that $ d^{*\iota}_{(1, M^*)} (x_i, y_i) = d^\iota_{(1, M^*)}(x_i, y_i)$. Apply Lemma~\ref{lem-segment-point} and the BK inequality (since we may ask that the paths in the gaps are disjoint) to control $d^{*\iota}_{(1, M^*)} (x_i, y_i) = d^\iota_{(1, M^*)}(x_i, y_i)$ and using \eqref{eq-referee-2}, we deduce that
\begin{equation}\label{eq-case-1}
\P(\|P^*\|_{1} \leq c_3(\gamma) d_{(1, M^*)}(0, M^*) t/M^*) \leq \mathrm{e}^{-c_4(\iota, C_2, \gamma) t/M^*}\,.
\end{equation}
Now, let us consider $P\in \tilde{\mathcal{P}}_t \setminus \tilde{\mathcal{P}}_{t, \kappa}$. In this case, let $(x_i, y_i)$ be the gaps in $P$ such that $|y_i-x_i| \geq \kappa M^*$. Therefore,
$$\|P^*\|_{1} \geq \sum_{i=1}^{\kappa t/M^*}  d^*_{(1, M^*)}(x_i, y_i)\,.$$
Applying \eqref{eq-point-whole-line-lower} and the BK inequality (analogous to the derivation of \eqref{eq-BK-referee}), we obtain that
$$\P(\|P^*\|_{1} \leq \epsilon c_\kappa \kappa d_{(1, M^*)}(0, M^*) t/M^*) \leq g_1(\epsilon)^{\kappa t/M^*}\,,$$
for some $g_1(\epsilon)\to 0$ as $\epsilon \to 0$. Combined with \eqref{eq-P-t-C} and \eqref{eq-case-1}, we obtain that (recall that $\E |\mathcal{P}_t|\leq C^{t/M^*}$)
\begin{align*}&\P(\exists P \in \mathcal{P}_{t}: \|P^*\|_{1} \leq (\epsilon c_\kappa \kappa \wedge c_3(\gamma))\kappa d_{(1, M^*)}(0, M^*) t/M^*)\\
 \leq & \mathrm{e}^{-c_2 t/M^*} +  \E |\tilde{\mathcal{P}}_{t, \kappa}| \mathrm{e}^{-c_4(\iota, C_2, \gamma) t/M^*} + \E |\tilde{\mathcal{P}}_t \setminus \tilde{\mathcal{P}}_{t, \kappa}| g_1(\epsilon)^{\kappa t/M^*} \\
\leq & \mathrm{e}^{-c_2 t/M^*} + \mathrm{e}^{-c_4(\iota, C_2, \gamma) t/M^*} \mathrm{e}^{C_1 \sqrt{\kappa} t/M^*} + C^{t/M^*}g_1(\epsilon)^{\kappa t/M^*}\,,\end{align*}
where in the last inequality we have plugged in \eqref{eq-P-t-kappa} and Lemma~\ref{lem-exploration}.
Choosing $\kappa$ sufficiently small and then choosing $\epsilon$ sufficiently small depending on $\kappa$ yields that
\begin{equation}\label{eq-almost-finish-super-multiplicative}
\P(\exists P\in \mathcal{P}_t :\|P^*\|_{1} \leq c^\star \E d_{(1, M^*)}(0, M^*) t/M^*) \leq 3\mathrm{e}^{-c^\star t/M^*}\,,\end{equation}
where $c^\star>0$ is a constant that depends on all the constants that appear so far (but independent of $t$ and $M^*$). Denote by $H(\cdot)$ the distribution function of $d_{(1, M)}(0, M)$. Note that $$M^*d_{(1, M)}(0, M) \stackrel{\mathrm{law}}{=} d_{(M^*, MM^*)}(0, MM^*)$$ by Lemma~\ref{lem-renormalization}. Combined with \eqref{eq-almost-finish-super-multiplicative}, an analogous derivation to \eqref{eq-referee-page-8-1} and \eqref{eq-referee-page-8-2} yields that
\begin{align*}
\E d_{(1, MM^*)}(0, MM^*) &\geq c^\star \E d_{(1, M^*)}(0, M^*) \int_0^\infty t  (1 - \sum_{i=0}^\infty 3 \mathrm{e}^{- c^\star 2^i t}) dH(t) \\
&\geq \hat c^\star \E d_{(1, M^*)}(0, M^*) \E d_{(1, M)}(0, M)\,,
\end{align*}
for a constant $\hat c^\star>0$, completing the proof of the lemma.
\end{proof}

Next, we prove the power law on the expected distance in long range percolation.
\begin{lemma}\label{lem-power-law}
For all $\beta>0$, there exist constant $C, c>0$ such that $c M^\theta \leq \E d_{(1, M)}(0, M) \leq CM^\theta$ for all $M>0$, where $\theta$ is from Proposition~\ref{prop-existence-exponent}.
\end{lemma}
\begin{proof}
By Lemma~\ref{lem-submultiplicative}, we know that there exists $C>0$ such that
\begin{equation}\label{eq-sub}
\E d_{(1, MM')}(0, MM') \leq C \E d_{(1, M)}(0, M) \E d_{(1, M')}(0, M') \mbox{ for all } M, M'\geq 1\,.
\end{equation}
By Lemma~\ref{lem-super-multiplicative}, there exists $\theta'<\theta$ and $c>0$ such that
\begin{equation}\label{eq-sup}
\E d_{(1, MM')}(0, MM') \geq c \E d_{(1, M)}(0, M) \E d_{(1, M')}(0, M')
\end{equation}
for all $M, M'\geq 1$ where $M$ is $(1,\theta')$-good. Now, take $\theta'<\theta^*<\theta$.  By \eqref{eq-sub}, the function $\psi^*(x) = \log \E d_{(1, \mathrm{e}^x)}(0, \mathrm{e}^x) - \theta^* x$ has an infinite number of record values, and thus there exist an infinite number of points which are $(1,\theta^*)$-good. Choose $m$ as the smallest $(1,\theta^*)$-good number such that $m^{\theta^* - \theta'} \geq (C\vee 1)/(c\wedge 1)$. For $x\geq 0$, define $\psi(x) = \log_m \E d_{(1, m^x)}(0, m^x) - \theta' x$. By \eqref{eq-sub} and \eqref{eq-sup}, we obtain that $\max_{x\in (m^k, m^{k+1})} \psi(x) \geq \max_{x\in (m^{k-1}, m^k)} \psi(x)$ for all $k\in \mathbb{N}$. Therefore, there exists $m_k\in (m^{k}, m^{k+1})$ for all $k\in \mathbb{N}$ such that $m_k$ is $(1,\theta')$-good (since there exist record values for $\psi(x)$ in each interval $(m^k, m^{k+1})$). Write $\theta_k = \log_{m_k} \E d_{(1, m_k)}(0, m_k)$. Using \eqref{eq-sub} and \eqref{eq-sup} again, we obtain that for a constant $C^\star>0$
$$|\theta_{k2^{n+1}} - \theta_{k 2^{n}}|\leq \tfrac{C^\star}{k 2^n}\,, \mbox{ for all }k, n\in \mathbb{N}\,.$$
Therefore, $|\theta_k - \theta| \leq \sum_{n=0}^\infty \tfrac{C^\star}{k 2^n} \leq 2C^\star/k$, which implies that for constants $c', C'>0$
$$c' m_k^\theta \leq  \E d_{(1, m_k)}(0, m_k) \leq C' m_k^{\theta}\,.$$
Combined with Corollary~\ref{cor-first-moment-does-not-decrease}, this completes the proof.
\end{proof}

Combining Lemmas~\ref{lem-second-moment-good} and \ref{lem-power-law}, we can bound the second moment of $d_{(1, M)}(0, M)$ for all $M>0$ as every $M$ is ``good''. Indeed, recursively adapting the proof of Lemma~\ref{lem-second-moment-good} yields the following bounds on higher moments (we do not reproduce the proof here as it is just a slight modification).
\begin{lemma}\label{lem-moments}
For $\beta>0$ and $k\in \mathbb{N}$, there exists a constant $C>0$ such that for all $M^* > 0$
$$\E (d_{(1, M^*)}(0, M^*))^k \leq C (\E d_{(1, M^*)}(0, M^*))^k\,.$$
\end{lemma}
We conclude this subsection by determining the scaling of of the continuous model.
\begin{proof}[Proof of Theorem~\ref{thm-cts}] The result on the typical distance that $d^*_{(1, n)}(0, n) \asymp_P d_{(1, n)}(0, n)   \asymp_P n^\theta$ is an immediate consequence of Lemmas~\ref{lem-distance-tight}, \ref{lem-power-law} and \ref{lem-moments}. In addition, the lower bound on the diameter trivially follows from the lower bound on the typical distance. It remains to establish the upper bound on the diameter, for which we apply a union bound. Choose $k^*\in \mathbb{N}$ such that $(4/3)^{k^*\theta} \geq 4$. Without loss of generality, we assume that $n = 2^{m}$ for $m\in \mathbb{N}$. Now, define events
$$\Omega_\ell = \cap_{j=0}^{2^{m-\ell}}\{d_{(1, n)}(j2^\ell, (j+1)2^\ell) \leq \alpha n^\theta (3/2)^{(\ell - m)\theta}\} \mbox{ and } \Omega = \cap_{\ell=0}^m \Omega_\ell\,.$$
By splitting up a path from $n_1$ to $n_2$ into dyadic segments, it is clear that on $\Omega$, we have
$$d_{(1, n)}(n_1, n_2)\leq  2 \alpha n^\theta \sum_{\ell = 0}^{m}(3/2)^{(\ell - m)\theta} \leq C_{\theta}\alpha n^\theta\,, \mbox{ for all } 1\leq n_1, n_2 \leq n\,,$$
where $C_\theta>0$ is a constant that depends only on $\theta$. Therefore, it remains to bound $\P(\Omega)$. By a simple union bound and an application of Markov's inequality with Lemma~\ref{lem-moments}, we obtain that
$$\P(\Omega^c) \leq \sum_{\ell = 1}^m 2^{m-\ell} \P(d_{(1, n)}(0, 2^\ell) > \alpha n^\theta (3/2)^{(\ell - m)\theta}) \leq \sum_{\ell=1}^m 2^{m-\ell} \tfrac{C_{ \beta, k^*}}{\alpha} 4^{\ell - m}\,,$$
where we have used the choice that $(4/3)^{k^*\theta} \geq 4$. Moreover, $C_{\beta, k^*}>0$ is a constant coming from Lemma~\ref{lem-moments} and it only depends on $\beta$ and $k^*$. Therefore, $\P(\Omega)\leq 2 C_{\beta, k^*}/\alpha$, completing the proof by sending $\alpha\to\infty$.
\end{proof}

\section{Coupling the discrete and continuous models}\label{s:ContinuousToDiscrete}
In this subsection, we show that the long range percolation model is well approximated by our continuous model, thereby providing a proof of Theorem~\ref{thm-1}.
\begin{proof}[Proof of Theorem~\ref{thm-1}] We write the proof for the distance $d_{\mathrm{LRP}}(0, n)$, while the proof for $d^*_{\mathrm{LRP}}(0, n)$ and the diameter follow in the same manner.
Fix $\beta>0$. We consider the discrete model $d_{\mathrm{LRP}}(1, n)$ and the continuous model $d_{(1, n)}(1, n)$.
For each pair $(k, \ell)$, define
\begin{align*}
E^{\mathrm{dsc}}_{k, \ell} &\deq \{ \mbox{ the edge }\langle k, \ell\rangle \mbox{ occurs in the discrete model}\}\,,\\
E^{\mathrm{cts}}_{k, \ell} &\deq \{ \mbox{an edge } \langle x, y\rangle \mbox{ occurs in the continuous model for }x\in(k-\tfrac{1}{2}, k+\tfrac{1}{2}) \mbox{ and } y\in (\ell-\tfrac{1}{2}, \ell+\tfrac{1}{2}) \}\,.
\end{align*}
Denote by $N_{k, \ell}$ the number of edges of form $(x, y)$ that occurs in the continuous model for $x\in(k-\tfrac{1}{2}, k+\tfrac{1}{2})$ and  $y\in (\ell-\tfrac{1}{2}, \ell+\tfrac{1}{2})$. We compute that
\begin{align}\label{eq-N-k-ell}
\E N_{k, \ell} &= \beta \int_{k-1/2}^{k+1/2} \int_{\ell - 1/2}^{\ell+1/2} \frac{1}{(x-y)^2}  dy dx \nonumber \\
&= \beta \int_{k-1/2}^{k+1/2}\left| \frac{1}{x - (\ell - \tfrac12)} - \frac{1}{x - (\ell + \tfrac12)}\right| dx = \beta \log \frac{(k-\ell)^2}{(k-\ell)^2-1}\,.
\end{align}
We can construct a coupling of the continuous and discrete models each of the pairs of events  $E^{\mathrm{dsc}}_{k, \ell}$ and $E^{\mathrm{cts}}_{k, \ell}$ in such a way that this collection of pairs is independent. For two probability measures on $\Omega$, denote the total variation distance between $\mu$ and $\nu$ by $\|\mu - \nu\|_{\mathrm{TV}} = \sup_{A\subset \Omega}|\mu(A) - \nu(A)|$. It is a well-known fact that there exists a coupling $(X, X')$ with $X\sim \mu$ and $X'\sim \nu$ such that $\P(X\neq X') = \|\mu - \nu\|_{\mathrm{TV}}$.  Using this fact and by \eqref{eq-N-k-ell} we could make the coupling such that
\begin{equation}\label{eq-coupling-c-d-prime}
\P(E^{\mathrm{dsc}}_{k, \ell} \setminus E^{\mathrm{dsc}}_{k, \ell}) + \P(E^{\mathrm{dsc}}_{k, \ell} \setminus E^{\mathrm{dsc}}_{k, \ell}) = \|\mathrm{Ber}(\tfrac{\beta}{(k-\ell)^2}) - \mathrm{Poi}(\beta \log \tfrac{(k-\ell)^2}{(k-\ell)^2-1})\|_{\mathrm{TV}} \leq  C_1(\beta + \beta^2) \tfrac{1}{(k-\ell)^4}\,,
\end{equation}
where we used $\mathrm{Ber}(x)$ and $\mathrm{Poi}(x)$ denote Bernoulli and Poisson variables with expectation $x$ for $x \geq 0$, and $C_1>0$ is an absolute constant.
Thus, we have
\begin{equation}\label{eq-coupling-c-d}
\P((E^{\mathrm{dsc}}_{k, \ell})^c \mid E^{\mathrm{cts}}_{k, \ell}) \leq C_2(\beta + \beta^2) \tfrac{1}{(k-\ell)^2}\,, \mbox{ and } \P((E^{\mathrm{cts}}_{k, \ell})^c \mid E^{\mathrm{dsc}}_{k, \ell}) \leq C_2(\beta + \beta^2) \tfrac{1}{(k-\ell)^2}\,,
\end{equation}
where $C_2>0$ is an absolute constant.

Now, consider a geodesic $P^*$ in $d_{\mathrm{LRP}}(1, n)$. We construct a path $P$ connecting $1$ and $n$ in the continuous model in accordance to geodesic $P^*$ such that
\begin{itemize}
\item  For $\langle k, \ell\rangle \in P^*$, if $E^{\mathrm{cts}}(k, \ell)$ occurs we use a corresponding edge for $P$ and walk on the underlying Euclidian line to incorporate the walk between endpoints.
\item For $\langle k, \ell\rangle \in P^*$, if $E^{\mathrm{cts}}(k, \ell)$ does not occur, then in $P$ we just walk on the underlying Euclidian line from $k$ to $\ell$.
\end{itemize}
By construction, we obtain that
$$\|P\|_1 \leq \|P^*\|_{\mathrm{LRP}} + \sum_{(k, \ell)\in P^*} |k-\ell| (1-\one_{E^{\mathrm{cts}}(k, \ell)})\,.$$
Using the aforementioned coupling between the discrete and the continuous model such that \eqref{eq-coupling-c-d} holds, we obtain that (denoting by $\mathcal F_{\mathrm{dis}}$  the $\sigma$-field generated by all the information in the discrete model)
\[
\E \Big[\sum_{\langle k, \ell\rangle\in P^*} |k-\ell| (1-\one_{E^{\mathrm{cts}}(k, \ell)}) \mid \mathcal F_{\mathrm{dis}}\Big] \leq \E \Big[ \sum_{\langle k, \ell\rangle\in P^*} |k-\ell| \frac{C_3(\beta + \beta^2)}{(k-\ell)^2}\mid \mathcal F_{\mathrm{dis}}\Big] \leq C_4(\beta + \beta^2)\|P^*\|_{\mathrm{LRP}}\,,
\]
where $C_3, C_4$ are two absolute constants. 
Hence for all $\epsilon>0$ there exists $C_\epsilon>0$ such that
$$\P(\|P\|_1 \leq C_\epsilon \|P^*\|_{\mathrm{LRP}}) \geq 1-\epsilon \,.$$
Since $ d_{(1, n)}(0, n) \leq \|P\|_1$ we have that,
\begin{equation}\label{eq-cts-dis}
\P(d_{(1, n)}(0, n) \leq C_\epsilon d_{\mathrm{LRP}}(0, n)) \geq 1-\epsilon \,.
\end{equation}

Next, let $P$ be the geodesic in $d_{(1, n)}(1, n)$ and let $h(P)$ be the number of hops in the geodesic.  We construct a path $P^*$ connecting $1$ and $n$ in the discrete model according to $P$ such that
\begin{itemize}
\item  For $\langle x, y\rangle \in P$ with $x\in (k-1/2, k+1/2)$ and $y\in (\ell - 1/2, \ell+1/2)$, if $E^{\mathrm{dsc}}(k, \ell)$ occurs we use a corresponding edge for $P^*$ and walk on the underlying Euclidian line to incorporate the deformation of the edge.
\item For $\langle k, \ell\rangle \in P$ with $x\in (k-1/2, k+1/2)$ and $y\in (\ell - 1/2, \ell+1/2)$, if $E^{\mathrm{dsc}}(k, \ell)$ does not occur, then in $P^*$ we just walk on the underlying Euclidian line from $k$ to $\ell$.
\end{itemize}
By construction, we obtain that
$$\|P^*\|_{\mathrm{LRP}} \leq d_{(1, n)}(1, n) + h(P) +  \sum_{(k, \ell) \in P^\star} |k-\ell| (1-\one_{E^{\mathrm{cts}}(k, \ell)})\,,$$
where $P^\star = \{(k, \ell): \mbox{ there exists } \langle x, y\rangle\in P \mbox{ with } x\in (k-1/2, k+1/2)  \mbox{ and } y\in (\ell - 1/2, \ell+1/2)\}$.
Applying \eqref{eq-coupling-c-d} and Lemma~\ref{lem-exploration} and  employing an analogous derivation to \eqref{eq-cts-dis}, we deduce that for all $\epsilon>0$ there exists $C_\epsilon>0$ such that
$$\P(\|P^*\|_{\mathrm{LRP}} \leq C_\epsilon \|P\|_1) \geq 1-\epsilon \,.$$
Since $d_{\mathrm{LRP}}(0, n) \leq \|P^*\|_{\mathrm{LRP}}$, we have that
$$\P(d_{\mathrm{LRP}}(0, n) \leq C_\epsilon d_{(1, n)}(0, n)) \geq 1-\epsilon \,.$$
Combined with \eqref{eq-cts-dis} and Theorem~\ref{thm-cts}, it gives that $d_{\mathrm{LRP}}(0, n)  \asymp_P  d_{(1,n)}(0, n)  \asymp_P n^\theta$.
which completes the proof.
\end{proof}

\end{document}